\documentclass[a4paper,10pt]{article}
\usepackage{amsmath,amssymb} 

\usepackage{amsfonts} 
\usepackage{amsthm} 
\usepackage{bm} 

\usepackage{mathrsfs} 

\DeclareMathAlphabet{\mathrsfs}{U}{rsfs}{m}{n} 
\usepackage{amscd}
\usepackage{amsmath}
\usepackage{bbding}
\usepackage{amsfonts}
\usepackage{cite}
\usepackage{mathrsfs}
\usepackage{stmaryrd}
\usepackage{graphicx,latexsym,amsmath,amssymb,amsthm,enumerate}
\usepackage{bm}
\usepackage{color}

\theoremstyle{definition}
\newtheorem{theorem}{Theorem}[section]
\newtheorem{lemma}{Lemma}[section]
\newtheorem{corollary}{Corollary}[section]
\newtheorem{proposition}{Proposition}[section]
\newtheorem{definition}{Definition}[section]

\newtheorem{remark}{Remark}[section]

\def\acknowledgement{\par\addvspace{17pt}\small\rmfamily
	\trivlist\if!\ackname!\item[]\else
	\item[\hskip\labelsep
	{\bfseries\ackname}]\fi}

\def\ackname{Acknowledgements}
\hfuzz=\maxdimen
\begin{document}
\title{{\LARGE\bf{Improved conditional gradient method for the generalized cone order optimization problem on the local sphere}}}
\author{Li-wen Zhou$^{a,b}$\footnote{E-mail address: zhouliwen@live.cn} , Ya-ling Yi $^{a}$, Min Tang $^{a}$
	and Yao-Jia Zhang$^{a,b}$\\
	$^a${\small\it School of Sciences, Southwest Petroleum University, Chengdu, 610500, P.R. China}\\
	$^b${\small\it Institute for Artificial Intelligence, Southwest Petroleum University, Chengdu 610500, Sichuan, P. R. China}\\
	}
\date{ }
\maketitle
\vspace*{-9mm}
\begin{center}
\begin{minipage}{5.6in}
\noindent{\bf Abstract.} 
In this paper, a generalized optimization problem on the local sphere is established by the cone order relation on the tangent space, and solved by an improved conditional gradient method (for short, ICGM). The auxiliary subproblems are constructed by the directed distance function on the tangent space, the iteration step size is updated by the Armijo rule, and the convergence of the ICGM is proved without the convexity of the objective function. Under the assumption of convexity, the clusters of the sequence generated by the ICGM are proved to be the spherical weakly Pareto solutions (also known as weakly efficient solutions) of this problem .
\\ \ \\
{\bf Keywords:} Generalized cone order vector optimization problem, stationary point, spherical weakly efficient solution, conditional gradient method
\\ \ \\
{\bf 2020 AMS Subject Classifications}: 49J40; 47J20.
\end{minipage}
\end{center}

\section{Introduction} \noindent

In recent years, the study of optimization problems involves many different types of problems, single-objective, multi-objective vector optimization problems, convex optimization problems, discrete optimization problems and so on. Among them, classical optimization problems study the function from $R^{n}$ to $R$, there are many classical algorithms, such as Newton's method, gradient descent method, conjugate gradient method, etc. Multi-objective Optimization Problem (MOP) studies the function from $R^{n}$ to $R^{m}$, refers to the problem that there are multiple objective functions that need to be optimized at the same time in the optimization process. In multi-objective optimization, multiple objective functions must be minimized simultaneously. Usually, there is no single point that minimizes all the given objective functions simultaneously, so according to the literature \cite{1}, the concept of optimality can be replaced by the concept of Pareto optimality or efficient.

In optimization problems, cone order relations and variational inequalities provide flexible and rich mathematical tools for algorithms, enabling optimization models to describe real problems more accurately and help find optimal solutions. They play an important role in the fields of convex optimization, nonlinear optimization, integer programming, and contribute to the development of optimization theory and methods.

Multi-objective optimization problems can be applied in engineering \cite{2}, statistics \cite{3}, management science \cite{4,5}, finance \cite{6}, environmental analysis \cite{7}, etc. There are roughly two types of algorithms for solving MOPs, one is the purified algorithm for direct computational solution represented by the genetic algorithm, and the other is the classical algorithm combining the scalarization method of MOPs with the traditional single objective optimization. In recent years, many iterative methods for solving scalarized optimization problems have been extended to the multi-objective optimization case, such as, the most rapid descent method \cite{9}, Newton's method \cite{1,10}, Quasi-Newton method \cite{11,12}, trust region method \cite{13}, and the sub-gradient method \cite{14}. It is worth mentioning that the vast majority of the above algorithms deal with vector optimization problems with non-negative orthogonal cone $R_{+}^{m}$ induced partial order relations, however, unlike them, the partial order relations in \cite{15,16,17} use not nonnegative orthogonal cone $R_{+}^{m}$ but rather closed convex cone. Such cones have recently been analyzed in real-world problems, e.g., portfolio selection problems in securities markets \cite{18,19}, vector approximation problems, and $n$-cooperator differential games \cite{20}. Therefore, it is necessary to focus attention on vector optimization problems.

The literature \cite{9} discussed the steepest descent method for weakly effective solutions of multi-objective optimization problems and the cone $C$ is a non-negative orthogonal cone of $R^{n}$; literature \cite{21} introduces a partial order method for the general cone in $R^{n}$; literature \cite{16} gives an extension of the projected gradient method to the case of constraint vector optimization, where the order is given by the general cone in $R^{n}$; literature \cite{22} provides a geometric interpretation of the weighting method for constraint vector optimization. Bonnel et al \cite{23} constructed a proximal endpoint algorithm to study the convex vector optimization problem, and finding the weakly efficient solutions from finite dimensional Hilbert space $X$ to the Banach space $Y$ mapping; Ceng and Yao \cite{24} investigated the proximal endpoint method and discussed an extension of the proximal endpoint algorithm based on the Bregman function for solving weakly effective solutions to vector optimization problems; furthermore, the Newton-like method investigated by Chuong \cite{25} in 2013 to find effective solutions for vector optimization problems of mappings from a finite-dimensional Hilbert space $X$ to the Banach space $Y$, the mappings are associated with the partial order caused by nonempty closed convex pointed cones $C$ and the auxiliary subproblem of the algorithm uses the well-known directed distance function; similarly, the conditional gradient method proposed by Chen and Yang \cite{26} for solving constrained vector optimization problems with respect to a partial order induced by a closed, convex and pointed cone with nonempty interior and the construction of the auxiliary subproblem is also based on the well-known oriented distance function. Therefore, the directed distance function is worth employing in solving vector optimization problems. In addition to this, Bonnel et al \cite{27}, Chuong \cite{28} and Botand \cite{29} have proposed some extensions for solving vector optimization problems in infinite dimensional settings.

Inspired by the above study, we consider generalizing the optimization problem to the local sphere. Since local spheres are manifolds and do not have proper ordinal relations, some definitions and theorems in linear space cannot be used directly, so we need to map the problem indirectly into linear space in order to utilize the concepts and techniques of Euclidean space to solve it. For example, ordinal relations are established through cones in vector optimization problems (Literature \cite{9,16,21,22,23,24,25,26,27,28,29}). Thus, we can define a new ordinal relation by cones on the tangent space of a local sphere.

The generalized cone-order optimization problem on a local sphere is a very interesting and worthwhile problem. Since the local sphere is smooth but it has bad convexity conditions and linear structure, using differential properties is a good choice for solving the problem, so we can improve and extend the multi-objective conditional gradient method in \cite{26} to solve the local sphere optimization problem.

The outline of this paper is as follows. In section 2, we introduce some important concepts, theorems, and propositions. In section 3, we introduce the generalized cone order optimization problem on the local sphere $M$, give its optimality condition and define its descent direction. In section 4, based on section 3, we establish our algorithm and proves its convergence.

\section{Preliminaries}\noindent
\setcounter{equation}{0}

Let $\langle , \rangle$ be the Euclidean inner product, and the corresponding paradigm is denoted by $\Vert . \Vert$. In this paper, we study unit ball, so the local sphere $M$ and its tangent space at the point $p$ are denoted respectively:
$$M\subset S^2=\{p=(p_1,p_2,p_3)\in R^3\colon\|p\|=1\}$$
and
$$T_pM=\{v\in R^2\colon\ \langle p,v\rangle=0, p\in M\}.$$

In addition to this. The tangent bundle of $M$ is denoted as:
$$TM=\bigcup_{x\in M}T_xM,$$
the vector field $V$ on $M$ is the mapping from $M$ to $TM$, each point $x \in M$ is associated with the vector $V(x) \in T_x M$ (see \cite{30}, p.22).

\begin{definition}(\cite{31}, p.2)
	The intrinsic distance on the two-dimensional localized sphere between two arbitrary points $p,q \in M$ is deﬁned by
	\begin{equation}\label{eq2.1}
	d(p,q){:}=arccos\langle p,q\rangle,
	\end{equation}
	it can be shown that the intrinsic distance $d(p,q)$ between two arbitrary points $p,q \in M$ is obtained by minimizing the arc length functional,
	$$l(\gamma){:}=\int_a^b\|\gamma^{\prime}(t)\|dt,$$
	over the set of all piecewise continuously differentiable curves $\gamma : [a, b]\rightarrow M$ joining $p$ to $q$, i.e., $x=\gamma(a)$ and $y=\gamma(b)$. Moreover, $d$ is a distance in $M$ and $(M,d)$ is a complete metric space.
\end{definition}

\begin{definition}
	A curve with constant curvature on a sphere is a geodesic.
\end{definition}

If the geodesic segment $\gamma : [a, b]\rightarrow M$ has an arc length equal to the intrinsic distance between its endpoints, i.e., if $l(\gamma)=\arccos\langle\gamma(a),\gamma(b)\rangle $ then it is said to be a minimum geodesic.

A vector field $V$ is said to be along a smooth curve $\gamma\in M$ parallel if and only if $\nabla_{\gamma^{\prime}}V=0$, where $\nabla$ is the Levi-Civita connection associated with $(M, \langle\cdot ,\cdot \rangle )$.

\begin{definition} \label{de2.3}
	We use $\nabla$ to introduce an isometry $P_{\gamma,\cdot,\cdot}$ on the tangent bundle $TM$ along $\gamma$, which is called the parallel transport and defined as:
	$$P_{\gamma,\gamma(b),\gamma(a)}(v){:}=V(\gamma(b)),\forall a,b\in R \text{,} v\in T_{\gamma(a)}M,$$
	where $V$ is the unique vector field satisfying $\nabla_{\gamma^\prime(t)}V=0$ for all $t$ and $V(\gamma(a))=v$. Without any confusion, $\gamma$ can be omitted when it is a minimal geodesic joining $x$ to $y$.
	\end{definition}
\begin{definition}(\cite{31}, p.3)
	The exponential mapping $exp_p : T_pM\to M$ is deﬁned by $exp_pv=r_v(1)$, where $r_{v}$ is the geodesic defined by the initial position $p$ and its velocity $v$, and $T_{p}M$ denotes $M$ the tangent space at $p$, which the expression is given:
	$$exp_pv:=\begin{cases}\cos(\|v\|)p+\sin(\|v\|)\frac v{\|v\|},v\in T_pM/\{0\},\\p,v=0,&\end{cases}$$
    and for all $t\in R$ have
	$$exp_ptv:=\begin{cases}\cos(\|tv\|)p+\sin(\|tv\|)\frac v{\|v\|},v\in T_pM/\{0\},\\p,v=0.&\end{cases} $$
\end{definition}
\begin{remark}
	We can also use the expression above for denoting the geodesic starting at the $p\in M$ with velocity $v\in T_pM$.
\end{remark}
\begin{definition}(\cite{31}, p.3)
	The inverse of the exponential mapping is given by the following equation
	\begin{equation}\label{eq2.2}
		exp_p^{-1}q:=\begin{cases}\frac{arccos\langle p,q\rangle}{\sqrt{1-\langle p,q\rangle^2}}(I-pp^T)q,q\notin\{p,-p\},\\0,q=p.&\end{cases}
	\end{equation}
	and by \eqref{eq2.1} and \eqref{eq2.2}:
	$$d(p,q){:}=\|exp_p^{-1}q\|,p,q\in M.$$
\end{definition}
\begin{proposition}(\cite{30}, p.149, Theorem 3.1) \label{pr2.1}
	$exp_p : T_pM\to M,\forall p\in M$ is a homomorphic mapping and for any two points $p,q\in M$, there exists a unique minimal geodesic $r_{p,q}(t)=exp_ptexp_p^{-1}q$, where $t\in[0,1]$ and the $exp_p^{-1}q$ denotes the vector obtained by projecting the geodesic line $r_{p,q}(t)$ in its tangent space.
\end{proposition}
\begin{proposition}(\cite{32}, p.150, Lemma 3.2) \label{pr2.2}
	The exponential mapping and its inverse are continuous on $M$.
\end{proposition}
\begin{definition}(\cite{30}, p.58)
	The subset $C\subset M$ is called geodesic convex if and only if for any two points $x,y\in C$, the geodesic connecting $x$ to $y$ is included in $C$, i.e., if $\gamma : [a, b]\rightarrow M$ is a geodesic such that $x=\gamma(a)$ and $y=\gamma(b)$, then for all $t\in[0,1],\gamma((1-t)a+tb)\in C $.
\end{definition}
\begin{remark}
	According to Proposition \ref{pr2.1}, on $M$, a subset $C\subset M$ is geodesically convex if and only if
	$$exp_xtexp_x^{-1}y\in C$$
	for all $x,y\in C,t\in[0,1]$.
\end{remark}
\begin{definition}
	A point set $K\subseteq\mathbb{R}^n$ is called a pointed cone if it satisfies
	\begin{itemize}
		\item[(1)] $a\in K, \lambda\geq0\Longrightarrow\lambda a\in K$;
		\item[(2)] $a\in K\mathrm{~and~}-a\in K\Longrightarrow a=0$.
	\end{itemize}
\end{definition}
\begin{definition} \label{de2.8}
	Let $p$ be a point on the sphere $S^{2}$, and $v$ be the tangent vector (unit vector) at the point $p$, so the directional derivative of the function $f : S^2\to S^2$ along the $v$ direction at point $p$ is defined as:
	$$\partial_vf(p){:}=\lim_{t\to0}\frac{f\left(\exp_p(tv)\right)-f(p)}t,$$
	where $\exp_{p}(tv)$ is the point that starts from point $p$ and travels $t$ unit distance on the sphere along the tangent vector $v$.
\end{definition}
\begin{definition}
	On the sphere $S^{2}$, we say that the function $f : S^2\to S^2$ is smooth if for any point $p$ on the sphere and any tangent vector $v$, the directional derivative $\partial_vf(p)$ exist and are continuous. In other words, the $f$ is a function that is everywhere differentiable on the sphere.
\end{definition}

Finally we recall the concept of directed distance function, which was first introduced by Hiriart-Urruty \cite{33} to analyze the geometric structure of a non-smooth optimization problem and to obtain the necessary optimality conditions, directed distance function has been widely used in several studies. Such as finding Lagrange multipliers for vector optimization problems and carving out various solutions of vector optimization problems \cite{34,35,36}, optimality conditions for D.C vector optimization problems \cite{37}, scalarization of vector optimization \cite{38}, robustness of multiobjective optimization \cite{39}, optimality conditions for vector optimization \cite{40} and set-valued optimization \cite{41,42} etc.
\begin{definition}
	Let $A$ be a subset of $R^{m}$, the functions $\triangle_A : {R}^m\to{R}$ defined below is called a directed distance function:
	$$\triangle_A(y){:}=d_A(y)-d_{R^m\setminus A}(y),\forall y\in R^m,$$
	among $d_A(y)=\inf\{\|y-a\|,a\in A\}$ refers to the point $y\in {R}^m$ to the set $A$ as a function of the distance to the set.
\end{definition}

\begin{lemma}(\cite{36}, Proposition 3.2) \label{le2.1}
	Define $\varphi_C(y)=\Delta_{-C}(y),\forall y\in R^m$. Then the next conclusion holds:
	\begin{itemize}
		\item[(i)] $\varphi_{\mathcal{C}}$ is a $1$-Lipschitzian function;
		\item[(ii)] For any $y\in int(-C)$, there is $\varphi_C(y)<0$; for any $y\in bd(-C)$, there are $\varphi_C(y)=0$; for any $y\in R^m\backslash(-C)=int(R^m\backslash(-C))$, there are $\varphi_C(y)>0$;
		\item[(iii)] $\varphi_{\mathcal{C}}$ is convex;
		\item[(iv)] $\varphi_{\mathcal{C}}$ is positively homogeneous;
		\item[(v)] For all $y_1,y_2\in\mathbb{R}^m$, $\varphi_C(y_1+y_2)\leq\varphi_C(y_1)+\varphi_C(y_2)$ and $\varphi_C(y_1)-\varphi_C(y_2)\leq\varphi_C(y_1-y_2)$;
		\item[(vi)] Let $y_1,y_2\in\mathbb{R}^m$, if $y_1\preccurlyeq_C(\prec_C)y_2$, then $\varphi_C(y_1)\leq(<)\varphi_C(y_2)$.
	\end{itemize}
\end{lemma}

\section{Generalized cone order optimization problem on local sphere $M$} \noindent
\setcounter{equation}{0}

In order to establish an optimization problem on the local sphere $M$, we need to define a new ordering relation.

Let $p$ be a given point in $M$, $C_{p}$ be a given cone on the tangent space $T_{p}M$, and the interior of cone $C_{p}$ is represented by $\operatorname{int}C_p$. Suppose that $C_{p}$ is a closed, convex and pointed cone with nonempty interior.
\begin{definition} \label{de3.1}
	For any $y\in M$, we define
	$$C_y:=P_{y,p}C_p:=\left\{P_{y,p}c_p\in T_yM:\forall c_p\in C_p\subset T_pM\right\},$$
	by the parallel transport from $p$ to $y$, and let
	$$\tilde{C}_y{:}=\{exp_yc_y:\forall c_y\in C_y\}$$
	be a subset of $M$ by the exponential mapping.
\end{definition}

It is worth noting that $C_{y}$ is a convex cone based on a given point $p\in M$ and the convex cone $C_p\subset T_pM$.
\begin{remark} \label{re3.1}
	From Definition \ref{de3.1} we know that for any point $x\in M$, in its corresponding tangent space $T_{x}M$, the given cone $C_{p}$ can be parallel transport to obtain a cone $C_x\subset T_xM$, and since the spherical curvature is fixed, the cone $C_{p}$ and the cone $C_{x}$ are isometric isomorphisms. 
\end{remark}

Now we can give the ordering relation between any two points $x,y\in M$ using the given cone $C_{p}$.

\begin{definition}
	Let $x,y$ be any given point on $M$, then
	$$y\preccurlyeq_{\tilde{C}_p}x$$
	if and only if $exp_x^{-1}y\in C_p$,
	besides,
	$$y\prec_{\tilde{C}_p}x$$
	if and only if $exp_x^{-1}y\in intC_p$.
\end{definition}
\begin{remark} \label{re3.2}
	It is easy to see that the ordering relation on $M$ is equivalent to the ordering relation on $T_{p}M$ relative to $C_{p}$. The latter is often mentioned in the field of vector optimization problems and vector variational inequalities (see \cite{43,44}). In other words,
	$$y\preccurlyeq_{\tilde{C}_p}x\Leftrightarrow0\preccurlyeq_{C_p}exp_x^{-1}y\Leftrightarrow exp_x^{-1}y\in C_p,$$
	$$y\prec_{\tilde{C}_p}x\Leftrightarrow0\prec_{C_p}exp_x^{-1}y\Leftrightarrow exp_x^{-1}y\in intC_p;$$
	and
	$$y\not\preccurlyeq_{\tilde{C}_p}x\Leftrightarrow0\not\preccurlyeq_{C_p}exp_x^{-1}y\Leftrightarrow exp_x^{-1}y\not\in C_p,$$
	$$y\not\prec_{\tilde{C}_p}x\Leftrightarrow0\not\prec_{C_p}exp_x^{-1}y\Leftrightarrow exp_x^{-1}y\not\in intC_p.$$
\end{remark}

If $F(x) : M\to M$ is a smooth function on a local sphere $M$, we will study the following generalized optimization problem (GOP) for a given convex cone $C_p\subset T_pM$,
\begin{equation} \label{wt3.1}
	\begin{aligned}&min_{C_p}F(x),\\&\mathrm\;{s.t.}\;x\in M.\end{aligned}
\end{equation}

Note that $min_{C_p}$ denotes the optimization model generated by the cone $C_p$ induced generalized ordering relation.
\begin{remark}
	If $M$ reduces to the Euclidean space $R^{n}$, then for all $x\in M$, there is $T_xM=M=R^n$, problem \eqref{wt3.1} degenerates into a vector optimization problem about cones in $R^{n}$, which is considered and studied in \cite{19}.
\end{remark}

Throughout the paper, we know that the ordering relation on $M$ is equivalent to the ordering relation on $T_pM$ relative to $C_p$. For the optimization problem \eqref{wt3.1}, let $y=F(x)$ then
$$y_{k+1}\preccurlyeq_{\tilde{C}_p}y_k\Leftrightarrow0\preccurlyeq_{C_p}exp_{y_k}^{-1}y_{k+1}\Leftrightarrow exp_{y_k}^{-1}y_{k+1}\in C_p,$$
$$y_{k+1}\prec_{\tilde{C}_p}y_k\Leftrightarrow0\prec_{C_p}exp_{y_k}^{-1}y_{k+1}\Leftrightarrow exp_{y_k}^{-1}y_{k+1}\in intC_p,$$
where $exp_{y_k}^{-1}y_{k+1}\in T_{y_k}M$ is a vector in the tangent space $T_{y_k}M$, we define the function
$$V_{y_k}:T_{y_k}M\longrightarrow T_{y_k}M,V_{y_k}(z)=-exp_{y_k}^{-1}z,$$
let
$$V_{y_k}(z_k)=0,V_{y_k}(z_{k+1})=-exp_{y_k}^{-1}y_{k+1},$$
so
$$V_{y_k}(z_k)-V_{y_k}(z_{k+1})=exp_{y_k}^{-1}y_{k+1},$$
where $V_{y_k}(z_k)$ and $V_{y_k}(z_k+1)$ represent the two endpoints of the vector $exp_{y_k}^{-1}y_{k+1}$. After organizing them, we obtain: If
$$y_{k+1}\preccurlyeq_{\tilde{C}_p}y_k,$$
then,
$$exp_{y_k}^{-1}y_{k+1}\in C_p,$$
$$V_{y_k}(z_k)-V_{y_k}(z_{k+1})\in C_p,$$
$$V_{y_k}(z_{k+1})\preccurlyeq_{C_p}V_{y_k}(z_k),$$
i.e.,
\begin{equation}
	y_{k+1}\preccurlyeq_{\tilde{C}_p}y_k\Leftrightarrow V_{y_k}(z_{k+1})\preccurlyeq_{C_p}V_{y_k}(z_k),
\end{equation}
similarly, it can be concluded that
\begin{equation}
	y_{k+1}\prec_{\tilde{C}_p}y_k\Leftrightarrow V_{y_k}(z_{k+1})\prec_{C_p}V_{y_k}(z_k),
\end{equation}
therefore, finding the iterative descent direction $d_k$ of function $V_{y_k}(z)$ at $z_k$ to get the next step $z_{k+1}$, which indirectly make $y_{k+1}\preccurlyeq_{\tilde{C}_p}y_k.$

To calculate the iteration direction $d_k$ in $T_{y_k}M$, we use the directional distance function.

In this paper, for convenience, we define $A=-C_{y_k}$ in Definition \ref{de2.8}. In the following we will use the function $\varphi_{C_{y_k}}:T_{y_k}M\longrightarrow R$, defined as follows:
$$\varphi_{C_{y_k}}\left(V_{y_k}(z)\right)=\Delta_{-C_{y_k}}\left(V_{y_k}(z)\right)=d_{-C_p}\left(V_{y_k}(z)\right)-d_{T_{y_k}M\setminus-C_p}\left(V_{y_k}(z)\right),\forall V_{y_k}(z)\in T_{y_k}M,$$
the function $\varphi_{\mathcal{C}_{\mathcal{Y}_k}}$ has all the properties described in Lemma \ref{le2.1}.

It is known that the function $V_{y_k}(z)=-exp_{y_k}^{-1}z$, which is easy to know $V_{y_k}:T_{y_k}M\longrightarrow T_{y_k}M$ is continuously differentiable from Definition \ref{de2.3} and Proposition \ref{pr2.2}, noting that its Jacobian is given by $JV_{y_k}(z)$, and we replace it by the following equation: $JV_{y_k}(z)=[\nabla v_1(z)\quad\nabla v_2(z)]^T,z\in T_{y_k}M$.

By Remark \ref{re3.1}, we know that the cone $C_{p}$ and the cone $C_{y_k}$ are isometric isomorphisms, so for the function $V_{y_k}(z)$, we have
$$y_{k+1}\preccurlyeq_{\tilde{C}_{y_k}}y_k\Leftrightarrow V_{y_k}(z_{k+1})\preccurlyeq_{C_{y_k}}V_{y_k}(z_k)\Leftrightarrow V_{y_k}(z_{k+1})\preccurlyeq_{C_p}V_{y_k}(z_k)\Leftrightarrow y_{k+1}\preccurlyeq_{\tilde{C}_p}y_k,$$
$$y_{k+1}\prec_{\tilde{C}_{y_k}}y_k\Leftrightarrow V_{y_k}(z_{k+1})\prec_{C_{y_k}}V_{y_k}(z_k)\Leftrightarrow V_{y_k}(z_{k+1})\prec_{C_p}V_{y_k}(z_k)\Leftrightarrow y_{k+1}\prec_{\tilde{C}_p}y_k,$$
thus, the original optimization problem \eqref{wt3.1} on the sphere is equivalent to the following vector optimization problem \eqref{wt3.4} on the tangent space: 
\begin{equation} \label{wt3.4}
	\begin{aligned}&min_{C_p}V_{y_k}(z),\\&\mathrm\;{s.t.}\;z\in\Omega.\end{aligned}
\end{equation}
where $\Omega\subset T_{y_k}M$ is a nonempty closed convex compact set.

\begin{definition} \label{de3.3}
	If the point $\overline{z}\in\Omega $ is a weakly efficient solution (or weakly Pareto solution) of problem \eqref{wt3.4}, then the $\overline{z}\ $ corresponding to $\overline{x}\in M$ is called a spherical weakly effective solution (or weakly Pareto solution) of problem \eqref{wt3.1}; if the point $\overline{z}\in\Omega $ is a efficient solution (or Pareto solution) of problem \eqref{wt3.4}, then the $\overline{z}\ $ corresponding to $\overline{x}\in M$ is called a spherically effective (or Pareto solution) of problem \eqref{wt3.1}. 
\end{definition}

By reference \cite{26}, we have a necessary, but not sufﬁcient, ﬁrst-order optimality condition for problem \eqref{wt3.4} at $z\in\Omega$, which is
$$JV_{y_k}(z)(\Omega-z)\bigcap(-intC_p)=\emptyset,$$
where
$$JV_{y_k}(z)(\Omega-z)=\begin{Bmatrix}JV_{y_k}(z)(s-z),s\in\Omega\end{Bmatrix},$$
and
$$JV_{y_k}(z)(s-z)=(<\nabla v_1(z),s-z>,<\nabla v_2(z),s-z>)^T,$$
clearly, this first-order optimality condition is equivalent to:
$$JV_{y_k}(z)(s-z)\notin-intC_p,\forall s\in\Omega.$$
\begin{definition}
	A sufficient necessary condition for problem \eqref{wt3.4}: $\overline{z}\in\Omega $ is a stationary point of problem \eqref{wt3.4} if and only if
	\begin{equation} \label{eq3.5}
		JV_{y_k}(\bar{z})(\Omega-\bar{z})\bigcap(-intC_p)=\emptyset.
	\end{equation}
\end{definition}
\begin{remark} \label{re3.4}
	From Lemma \ref{le2.1} (ii) it follows that if $\overline{z}\in\Omega $ is the stationary point of problem \eqref{wt3.4}, then
	$$\varphi_{C_{y_k}}\left(JV_{y_k}(\bar{z})(s-\bar{z})\right)\geq0,\forall s\in\Omega.$$
\end{remark}
\begin{proposition} \label{pr3.1}
If $\overline{z}\in\Omega $ is not the stationary point of problem \eqref{wt3.4}, then there exists a $\overline{s}\in\Omega $ such that $\bar{s}-\bar{z}$ is a decreasing direction of $V_{\mathcal{Y}_k}$ at $\bar{z}$.
\end{proposition}
\begin{proof}
	Because $\overline{z}\in\Omega $ is not a stationary point of problem \eqref{wt3.4}, then there exists a $\overline{s}\in\Omega $ such that
	$$JV_{y_k}(\bar{z})(\bar{s}-\bar{z})\in-intC_p,$$
	i.e.,
	$$JV_{y_k}(\bar{z})(\bar{s}-\bar{z})=\lim_{t\to0}\frac{V_{y_k}(\bar{z}+t(\bar{s}-\bar{z}))-V_{y_k}(\bar{z})}t\prec_{C_p}0, \bar{s}\in\Omega, \forall t>0,$$
	therefore, $\bar{s}-\bar{z}$ is a descent direction of $V_{\mathcal{Y}_k}$ at $\bar{z}$, we have
	$$V_{y_k}(\bar{z}+t(\bar{s}-\bar{z}))\prec_{C_p}V_{y_k}(\bar{z}),\forall t>0.$$
\end{proof}
\begin{theorem} \label{th3.1}
	\begin{itemize}
		\item[(i)] If $\overline{z}\in\Omega $ is a weakly effective solution of problem \eqref{wt3.4}, then $\overline{z}\in\Omega $ is a stationary point;
		\item[(ii)] If $V_{\mathcal{Y}_k}$ is $C_p$-convex on $\Omega$ and $\overline{z}\in\Omega $ is a stationary point of problem \eqref{wt3.4}, then $\overline{z}\in\Omega $ is a weakly effective solution.
	\end{itemize}
\end{theorem}
\begin{proof}
	\begin{itemize}
		\item[(i)] Suppose that $\overline{z}\in\Omega $ is a weakly effective solution of problem \eqref{wt3.4}. If $\bar{z}$ is not a stationary point, then \eqref{eq3.5} does not hold, according to Proposition \ref{pr3.1} can know, there exists a $z\in\Omega $ such that $z-\bar{z}$ is a descent direction of $V_{\mathcal{Y}_k}$, so
		$$V_{y_k}(\bar{z}+t(z-\bar{z}))\prec_{C_p}V_{y_k}(\bar{z}),\forall t>0,$$
		this is in contrast to $\bar{z}$ is a weakly effective solution of problem \eqref{wt3.4}. Therefore (i) holds. 
		\item[(ii)] Take any $z\in\Omega $, since $\bar{z}$ is a stationary point of problem \eqref{wt3.4}, according to Remark \ref{re3.4} there are
		$$\varphi_{C_{y_k}}\left(JV_{y_k}(\bar{z})(z-\bar{z})\right)\geq0,$$
		i.e.,
		$$JV_{y_k}(\bar{z})(z-\bar{z})\notin-intC_p,$$
		$$JV_{y_k}(\bar{z})(z-\bar{z})\succcurlyeq_{C_p}0.$$
		
		Due to $V_{\mathcal{Y}_k}$ is $C_{p}$-convex on $\Omega$, we have
		$$V_{y_k}(z)\succcurlyeq_{C_p}JV_{y_k}(\bar{z})(z-\bar{z})+V_{y_k}(\bar{z})\succcurlyeq_{C_p}V_{y_k}(\bar{z}),\forall z\in\Omega $$
		Therefore, $\bar{z}$ is a weakly efficient solution of problem \eqref{wt3.4}.		
	\end{itemize}
\end{proof}

\section{Improved algorithm for generalized cone order optimization problem on local sphere $M$}\noindent
\setcounter{equation}{0}

In this paper, we cannot directly calculate the descent direction of every point $y_k$. Instead, we search for the corresponding function's descent direction on the tangent space at this point to achieve the descent goal of $y_k$. For example, for $y_k$, to find its next step ${y}_{k+1}$, we need to define a smooth function $V_{\mathcal{Y}_k}$ in the tangent space $T_{y_k}M$, and then use the conditional gradient algorithm for vector optimization in reference \cite{26} to calculate the descent direction, such that we get $V_{y_k}(z_{k+1})\preccurlyeq_{C_p}V_{y_k}(z_k)$, which can be obtained: $y_{k+1}\preccurlyeq_{\tilde{C}_p}y_k$.

In the tangent space $T_{y_k}M$, for any given $z\in\Omega $, we introduce a useful auxiliary function $\Psi_Z : \Omega\to R$,
\begin{equation}
	\Psi_z(s)=\varphi_{C_{y_k}}\left(JV_{y_k}(z)(s-z)\right).
\end{equation}

To obtain the descent direction of $V_{y_k}$ at $z_k$, we need to consider the following auxiliary scalar optimization problem:
\begin{equation} \label{wt4.2}
    \min_{s\in\Omega}\Psi_Z(s).
\end{equation}

According to Lemma \ref{le2.1}, we know that the above $\Psi_{Z}(s)$ is a convex and continuous function, combined with the fact that $\Omega$ is a compact set, we know that problem \eqref{wt4.2} must have an optimal solution. The optimal solution and optimal value of problem \eqref{wt4.2} are defined as $s(z)$ and $\theta(z)$, respectively, i.e.,
\begin{equation} \label{eq4.3}
	s(z)=arg\min_{s\in\Omega}\Psi_z(s),
\end{equation}
\begin{equation} \label{eq4.4}
	\theta(z)=\min_{s\in\Omega}\Psi_z(s)=\Psi_z(s(z)).
\end{equation}

According to Proposition \ref{pr3.1}, we obtain the search direction of the objective function $V_{\mathcal{Y}_k}$ at $z$, defined as follows:
\begin{definition}
	For any given $z\in\Omega $, the search direction of the conditional gradient method for $V_{\mathcal{Y}_k}$ at $z$ is defined as:
	$$d(z)=s(z)-z,$$
	where $s(z)$ is generated by \eqref{eq4.3}.
\end{definition}

The following are some properties related to the stationary point of $\theta(z)$.

\begin{proposition} \label{pr4.1}
	Let $\theta:\Omega\to R$ be the function defined by \eqref{eq4.4}. Then, the following conclusion holds:
	\begin{itemize}
		\item[(i)] For any $z\in\Omega $, we have $\theta(z)\leq0$;
		\item[(ii)] $z\in\Omega $ is a stationary point of problem \eqref{wt3.4} if and only if $\theta(z)=0$.
	\end{itemize}
\end{proposition}
\begin{proof}
	\begin{itemize}
		\item[(i)] For any $z\in\Omega $, by means of \eqref{eq4.3} with \eqref{eq4.4} and $\varphi_{C_{y_k}}(0)=0$, we have
		$$\theta(z)=\min_{s\in\Omega}\Psi_z(s)\leq\Psi_z(z)=\varphi_{C_{y_k}}\left(JV_{y_k}(z)(z-z)\right)=\varphi_{C_{y_k}}(0)=0.$$ 
		\item[(ii)] Necessity. If $z\in\Omega $ is a stationary point of problem \eqref{wt3.4}, then according to Remark \ref{re3.4}: for any $s\in\Omega $, there is
		$$\varphi_{C_{y_k}}\left(JV_{y_k}(z)(s-z)\right)\geq0.$$
		From \eqref{eq4.3} we know that $s(z)\in\Omega $, and so:
		$$\theta(z)=\Psi_z(s(z))=\varphi_{C_{y_k}}\left(JV_{y_k}(z)(s(z)-z)\right)\geq0.$$
		By combining the above equation with (i), we can obtain:
		$$\theta(z)=0.$$
		
		Sufficiency. Let $\theta(z)=0$, according to \eqref{eq4.4}, we get
		$$0=\theta(z)\leq\Psi_z(s)=\varphi_{C_{y_k}}\left(JV_{y_k}(z)(s-z)\right),\forall s\in\Omega $$
		$$\Longrightarrow JV_{y_k}(z)(s-z)\bigcap(-intC_p)=\emptyset $$
		i.e.: $z$ is a stationary point of problem \eqref{wt3.4}.
	\end{itemize}
\end{proof}
\begin{remark}
	It follows from Proposition \ref{pr4.1} that $z\in\Omega $ is not a stationary point of problem \eqref{wt3.4} if and only if $\theta(z)<0$.
\end{remark}
\begin{corollary} \label{co4.1}
	If $\theta(z)<0$,$\forall z\in\Omega $, then for any $\beta\in(0,1)$, there exists $\hat{t}\in(0,1]$ such that
	$$V_{y_k}(z+td(z))\prec_{C_p}V_{y_k}(z)+\beta tJV_{y_k}(z)d(z),\forall t\in(0,\hat{t})$$
	where $d(z)=s(z)-z$, $s(z)$ is defined by \eqref{eq4.3}.
\end{corollary}
\begin{proof}
	Because $V_{\mathcal{Y}_k}$ is differentiable, we have
	$$V_{y_k}\big(z+td(z)\big)=V_{y_k}(z)+t\big(JV_{y_k}(z)d(z)+R(t)\big)$$
	where $\lim_{t\to0}R(t)=0$.
	From $\theta(z)<0$ we know that $\varphi_{C_{y_k}}\left(JV_{y_k}(z)(s(z)-z)\right)<0\Rightarrow JV_{y_k}(z)(s(z)-z)\in-intC_p$. This implies that for any $\beta\in(0,1),\quad(1-\beta)JV_{y_k}(z)(s(z)-z)\in-intC_p$. By the definition of open set, we have that there exists a $\hat{t}\in(0,1]$ such that, for $\forall t\in(0,\hat{t}],\quad\|R(t)\|$ is sufficiently small such that $R(t)+(1-\beta)JV_{y_k}(z)(s(z)-z)\in-intC_p$. This is equivalent to
	$$R(t)+(1-\beta)JV_{y_k}(z)(s(z)-z)\prec_{C_p}0$$
	$$\Leftrightarrow R(t)+JV_{y_k}(z)(s(z)-z)\prec_{C_p}\beta JV_{y_k}(z)(s(z)-z),$$
	therefore,
	$$V_{y_k}(z+td(z))=V_{y_k}(z)+t\big(JV_{y_k}(z)d(z)+R(t)\big)\prec_{C_p}V_{y_k}(z)+t\beta JV_{y_k}(z)d(z),\forall t\in(0,\hat{t}].$$
\end{proof}
\begin{proposition} \label{pr4.2}
	Let $\theta:\Omega\to R$ be the function defined by \eqref{eq4.4}, then $\theta$ is continuous on $\Omega$.
\end{proposition}
\begin{proof}
	Take $\overline{z}\in\Omega $ and let $\left\{z_{k}\right\}$ be a sequence on $\Omega$ such that $\lim_{k\to\infty}z_k=\bar{z}$. To obtain the continuity of $\theta$ on $\Omega$, it is only necessary to prove that: $\lim_{k\to\infty}\theta(z_k)=\theta(\bar{z})$, i.e.,
	\begin{equation}
		\lim_{k\to\infty}\sup\theta(z_k)\leq\theta(\bar{z})\leq\lim_{k\to\infty}\inf\theta(z_k).
	\end{equation}
	
	Since $s(\bar{z})\in\Omega $, by \eqref{eq4.3} and \eqref{eq4.4} we have that for all $k$,
	\begin{equation} \label{ie4.6}
		\theta(z_k)=\min_{s\in\Omega}\Psi_{z_k}(s)\leq\Psi_{z_k}(s(\bar{z}))=\varphi_{C_{y_k}}\left(JV_{y_k}(z_k)(s(\bar{z})-z_k)\right).
	\end{equation}
	
	Since $V_{\mathcal{Y}_k}$ is continuously differentiable and Lemma \ref{le2.1} (i) states that the function $\varphi_{\mathcal{C}_{\mathcal{Y}_k}}$ is continuous, so take simultaneously $\lim_{k\to\infty}sup$ on both sides of inequality \eqref{ie4.6}, we have
	$$\lim_{k\to\infty}\sup\theta(z_k)\leq\lim\sup_{k\to\infty}\varphi_{Cy_k}\left(JV_{y_k}(z_k)(s(\bar{z})-z_k)\right)=\varphi_{Cy_k}\left(JV_{y_k}(\bar{z})(s(\bar{z})-\bar{z})\right)=\theta(\bar{z}).$$
	
	Let us show that $\theta(\bar{z})\leq\lim_{k\to\infty}\inf\theta(z_k)$. Obviously, we have:
	\begin{align}
		\theta(\bar{z}) & = \min_{s \in \Omega} \Psi_{\bar{z}}(s) \leq \Psi_{\bar{z}}\big(s(z_k)\big) \nonumber \\
		& = \varphi_{C_{y_k}}\big(JV_{y_k}(\bar{z})(s(z_k)-\bar{z})\big) \nonumber \\
		& = \varphi_{C_{y_k}}\left(JV_{y_k}(\bar{z})(s(z_k)-z_k+z_k-\bar{z})\right) \nonumber \\
		& = \varphi_{C_{y_k}}\left(JV_{y_k}(\bar{z})\Big(s(z_k)-z_k\Big) + JV_{y_k}(\bar{z})(z_k-\bar{z})\right) \nonumber \\
		& \leq \varphi_{C_{y_k}}\left(JV_{y_k}(\bar{z})(s(z_k)-z_k)\right) + \varphi_{C_{y_k}}\left(JV_{y_k}(\bar{z})(z_k-\bar{z})\right).
	\end{align}

	Taking simultaneously $\lim_{k\to\infty}inf$ on both sides of the above inequality, we get
	\begin{align}
		\lim\inf_{k\to\infty}\theta(\bar{z}) & = \theta(\bar{z}) \leq \lim\inf_{k\to\infty} \varphi_{C_{y_k}}\left(JV_{y_k}(\bar{z})(s(z_k)-z_k)\right) + \lim\inf_{k\to\infty} \varphi_{C_{y_k}}\left(JV_{y_k}(\bar{z})(z_k-\bar{z})\right) \nonumber \\
		& \leq \lim\inf_{k\to\infty} \varphi_{C_{y_k}}\left(JV_{y_k}(\bar{z})(s(z_k)-z_k)\right) \nonumber \\
		& = \lim\inf_{k\to\infty}\left(\theta(z_k) + \varphi_{C_{y_k}}\left(JV_{y_k}(\bar{z})(s(z_k)-z_k)\right) - \varphi_{C_{y_k}}\left(JV_{y_k}(z_k)(s(z_k)-z_k)\right)\right) \nonumber \\
		& \leq \lim\inf_{k\to\infty}\left(\theta(z_k) + \left\|JV_{y_k}(\bar{z})(s(z_k)-z_k) - JV_{y_k}(z_k)(s(z_k)-z_k)\right\|\right) \nonumber \\
		& = \lim\inf_{k\to\infty}\left(\theta(z_k) + \left\|(JV_{y_k}(\bar{z}) - JV_{y_k}(z_k))(s(z_k)-z_k)\right\|\right) \nonumber \\
		& \leq \lim\inf_{k\to\infty}\left(\theta(z_k) + \left\|JV_{y_k}(\bar{z}) - JV_{y_k}(z_k)\right\|\|s(z_k)-z_k\|\right),
		\label{ie4.8}
	\end{align}
	because $s(z_k),\quad z_k\in\Omega $ and knowing that $\Omega $ is a tight set, so $\|s(z_k)-z_k\|\leq diam(\Omega)<\infty $. This, combined with the first order continuously differentiability of $V_{\mathcal{Y}_k}$, we get $\lim_{k\to\infty}JV_{y_k}(z_k)=JV_{y_k}(\bar{z})$, and so $\|JV_{y_k}(\bar{z})-JV_{y_k}(z_k)\|\|s(z_k)-z_k\|\to0(k\to\infty)$. The above conclusion combined with \eqref{ie4.8} is given:
	$$\theta(\bar{z})\leq\lim\inf_{k\to\infty}\theta(z_k).$$
	Therefore: $\theta$ is continuous on $\Omega $.
\end{proof}

Based on the conclusions given earlier, we can describe the process of improving the conditional gradient method for solving the problem \eqref{wt3.4}.

\textbf{Improved conditional gradient algorithm}\noindent
\begin{itemize}
\item[1] (Initialization) Choose the initial point $x_{0}$ , get the corresponding $y_{0}$ and its tangent space $T_{y_0}M$, let $k\leftarrow0$;
\item[2] In the tangent space $T_{y_k}M$, define the function $V_{y_k}(z)=-exp_{y_k}^{-1}z$;
\item[3] Solve the problem \eqref{wt4.2} to obtain $s(z_k)$ and $\theta(z_k)$ as in \eqref{eq4.3} and \eqref{eq4.4};
\item[4] If $\theta(z_k)=0$, then stop. Otherwise, proceed to the next step, Armijo line search;
\item[5] (Armijo line search) Let $\beta\in(0,1)$ and $\delta\in(0,1)$, choose the maximum $t_k$ from $J=\{\delta^n|n=0,1,2,...\}$ that satisfies \eqref{eq4.9} as the step size:
\begin{equation} \label{eq4.9}
	V_{y_k}(z_k+t_kd(z_k))\preccurlyeq_{C_p}V_{y_k}(z_k)+\beta t_kJV_{y_k}(z_k)d(z_k)
\end{equation}
\item[6] Define $z_{k+1}=z_k+t_kd(z_k),t_k\in(0,1]$. Make $k+1\leftarrow k$, then go back to step 3.
\end{itemize}

At the $k$-th iteration, we solve problem \eqref{wt4.2} with $z=z_k$. Let $s(z_k)$ and $\theta(z_k)$ denote respectively the optimal solution and the optimal value of problem \eqref{wt4.2} at $k$-th iteration. The direction of descent at $k$-th iteration is calculated by $d(z_k)=s(z_k)-z_k$. If $\theta(z_k)\neq0$, then we can use $d(z_k)$ with a step size strategy to look for a new solution $z_{k+1}$ which improves $z_k$. For convenience, in the text we can use $s_k,\theta_k,d_k$ instead of $s(z_k),\theta(z_k),d(z_k)$.

\begin{remark} \label{re4.2}
	Because $z_0\in\Omega,s_k\in\Omega, t_k\in(0,1]$, and $\Omega $ is convex set, it follows by induction that $z_k\in\Omega, \forall k\in Z$, i.e., the sequence $\{z_k\}$ generated by the improved conditional gradient algorithm is contained in $\Omega $.
\end{remark}

\section{Convergence analysis}\noindent
\setcounter{equation}{0}

It is worth noting that in the above algorithm we are placing the point $y_{k}$ into its corresponding tangent space $T_{y_k}M$ indirectly to find its descent direction, but in our convergence analysis, we do not need to consider the tangent spaces corresponding to other points. We explain reasons for this in Remark \ref{re3.1} and Remark \ref{re3.2}. Therefore, the descent direction and step size of any point on the tangent space in the algorithm can find elements with the same properties on its isomorphic tangent space, and there is no need to consider the corresponding tangent space for each step.

In this section, we analyze the convergence of the above algorithms.

Below we present some properties related to the iteration points of the improved conditional gradient algorithm with the Armijo step condition \eqref{eq4.9}.

\begin{proposition}
	For all $k$, we have
	\begin{itemize}
	\item[(i)] $V_{y_k}(z_{k+1})\prec_{C_p}V_{y_k}(z_k);$
	\item[(ii)] $t_k|\theta(z_k)|\leq\beta^{-1}\left(\varphi_{C_{y_k}}\left(V_{y_k}(z_{k+1})\right)-\varphi_{C_{y_k}}\left(V_{y_k}(z_k)\right)\right).$
	\end{itemize}
\end{proposition}
\begin{proof}
	\begin{itemize}
	\item[(i)] It is known that $z_{k}$ is not a stationary point, so $JV_{y_k}(z_k)d(z_k)\prec_{C_p}0$, and by \eqref{eq4.9} it follows that
	$$V_{y_k}(z_{k+1})=V_{y_k}(z_k+t_kd(z_k))\prec_{C_p}V_{y_k}(z_k).$$
	\item[(ii)] For any $k$, by \eqref{eq4.9} and Lemma \ref{le2.1} (iv)-(vi), we have
	\begin{align}
		\varphi_{C_{y_k}}\left(V_{y_k}(z_{k+1})\right) & \leq \varphi_{C_{y_k}}\left(V_{y_k}(z_k)+\beta t_kJV_{y_k}(z_k)d(z_k)\right) \nonumber \\
		& \leq \varphi_{C_{y_k}}\left(V_{y_k}(z_k)\right) + \varphi_{C_{y_k}}\left(\beta t_kJV_{y_k}(z_k)d(z_k)\right) \nonumber \\
		& = \varphi_{C_{y_k}}\left(V_{y_k}(z_k)\right) + \beta t_k\varphi_{C_{y_k}}\left(JV_{y_k}(z_k)d(z_k)\right) \nonumber \\
		& = \varphi_{C_{y_k}}\left(V_{y_k}(z_k)\right) + \beta t_k\theta(z_k).
		\label{ie5.1}
	\end{align}
	It is known that $z_{k}$ is not a stationary point, then $\theta(z_k)<0$; and combined with \eqref{ie5.1}, we get
	$$t_k|\theta(z_k)|=-t_k\theta(z_k)\leq\frac{\varphi_{C_{y_k}}\left(V_{y_k}(z_{k+1})\right)-\varphi_{C_{y_k}}\left(V_{y_k}(z_k)\right)}\beta.$$
	\end{itemize}
\end{proof}
\begin{theorem} \label{th5.1}
	Let $\{z_k\}$ be the sequence generated by the improved conditional gradient algorithm with Armijo step condition \eqref{eq4.9}. Then, every cluster point of $\{z_k\}$ is a stationary point of problem \eqref{wt3.4}.
\end{theorem}
\begin{proof}
	Let $\hat{z}\in T_{y_k}M$ be a cluster point of the sequence $\{z_k\}$, by Remark \ref{re4.2} we have $z_k\in\Omega$, combining with the closedness of $\Omega$, know that $\hat{z}\in\Omega$. So, there exists a subsequence $\left\{z_{k_j}\right\}$ of $\{z_k\}$, such that
	\begin{equation} \label{eq5.2}
	\lim_{j\to\infty}z_{k_j}=\hat{z}.
	\end{equation}
	
	By Propositions \ref{pr4.2} and \eqref{eq5.2}, we have that $\theta\left(z_{k_j}\right)\to\theta(\hat{z}),j\to\infty $. Therefore, by Proposition \ref{pr4.1} (ii), we only need to show that $\theta(\hat{z})=0$.
	Let $k=k_j$ in Proposition \ref{pr4.1} (ii), so
	\begin{equation} \label{ie5.3}
		0\leq t_{k_j}\left|\theta\left(z_{k_j}\right)\right|\leq\frac{\varphi_{Cy_k}\left(V_{y_k}\left(z_{k_{j+1}}\right)\right)-\varphi_{Cy_k}\left(V_{y_k}\left(z_{k_j}\right)\right)}{\beta}.
	\end{equation}
	
	It is known that $V_{\mathcal{Y}_k}$ is continuous and by Lemma \ref{le2.1} (i) we have the function $\varphi_{\mathcal{C}_{\mathcal{Y}_k}}$ is continuous, so $\varphi_{C_{y_k}}\left(V_{y_k}\left(z_{k_{j+1}}\right)\right)-\varphi_{C_{y_k}}\left(V_{y_k}\left(z_{k_j}\right)\right)\to0,j\to\infty $; combining with  \eqref{ie5.3} have: $\lim_{j\to\infty}t_{k_j}\left|\theta\left(z_{k_j}\right)\right|=0$, i.e.,
	\begin{equation}  \label{eq5.4}
		\lim_{j\to\infty}t_{k_j}\theta\left(z_{k_j}\right)=0.
	\end{equation}
	
	Since $t_k\in(0,1],\forall k$, we have the following two alternatives:
	\begin{equation} \label{ie5.5}
		\begin{array}{ccc}
			\mathrm{(a)}~\lim\limits_{j \to \infty} \sup t_{k_j}>0 ~or~ 	\mathrm{(b)}~\lim\limits_{j \to \infty} \sup t_{k_j}=0.
		\end{array}
	\end{equation}
	
	First, we suppose that  \eqref{ie5.5} (a) holds. Then, there exists a subsequence $\left\{t_{k_{j_i}}\right\}$ of $\left\{t_{k_j}\right\}$ converging to some $\hat{t}>0$. By  \eqref{eq5.2} we have: $\lim_{i\to\infty}z_{k_{j_i}}=\hat{z}$; and from  \eqref{eq5.4}, we know $\lim_{i\to\infty}t_{k_{j_i}}\theta\left(z_{k_{j_i}}\right)=0$, thus,
	$$\lim_{i\to\infty}\theta\left(z_{k_{j_i}}\right)=\theta(\hat{z})=0.$$
	
	Now, we assume that  \eqref{ie5.5} (b) holds. It is known that $\hat{z}\in\Omega$, by Proposition  \ref{pr4.1} (i) we have:
	\begin{equation} \label{ie5.6}
		\theta(\hat{z})\leq0.
	\end{equation}
	
	Obviously, $z_{k_j},s(z_{k_j})\in\Omega $ and
	$$\begin{Vmatrix}d(z_{k_j})\end{Vmatrix}=\begin{Vmatrix}s\left(z_{k_j}\right)-z_{k_j}\end{Vmatrix}\leq diam(\Omega)<\infty $$
	for all $j$, which implies that the sequence $\left\{d(z_{k_j})\right\}$ is bounded. Thus, we can take the subsequence $\left\{z_{k_{j_i}}\right\},\quad\left\{d(z_{k_{j_i}})\right\}$ and $\left\{t_{k_{j_i}}\right\}$, which converge to  $\hat{z}\mathrm{~,~}d(\hat{z})$ and $0$, respectively.
	Then, we take some ﬁxed but arbitrary $l\in N$ where $N$ denotes the set of natural numbers. For $\lim_{i\to\infty}t_{k_{j_i}}=0$, we have $t_{k_{j_i}}<\delta^l$ for $i$ large enough. This means that the Armijo condition \eqref{eq4.9} is not satisﬁed at $z=z_{k_{j_i}}$ for $t=\delta^l$, that is,
	$$V_{y_k}\left(z_{k_{j_i}}+\delta^ld\left(z_{k_{j_i}}\right)\right)\not\preccurlyeq_{C_p}V_{y_k}\left(z_{k_{j_i}}\right)+\beta\delta^lJV_{y_k}\left(z_{k_{j_i}}\right)d\left(z_{k_{j_i}}\right)$$
	or, equivalently,
	$$V_{y_k}\left(z_{k_{j_i}}+\delta^ld\left(z_{k_{j_i}}\right)\right)-V_{y_k}\left(z_{k_{j_i}}\right)-\beta\delta^lJV_{y_k}\left(z_{k_{j_i}}\right)d\left(z_{k_{j_i}}\right)\notin-C_p$$
	which means that
	\begin{equation} \label{eq5.7}
		V_{y_k}\left(z_{k_{j_i}}+\delta^ld\left(z_{k_{j_i}}\right)\right)-V_{y_k}\left(z_{k_{j_i}}\right)-\beta\delta^lJV_{y_k}\left(z_{k_{j_i}}\right)d\left(z_{k_{j_i}}\right)\in T_{y_k}M/(-C_p)=int(T_{y_k}M/(-C_p))
	\end{equation}
	where the last of equality holds in view of the closedness of $C_{p}$, by  \eqref{eq5.7} and Lemma \ref{le2.1} (ii), we have: 
	\begin{equation} \label{ie5.8}
		\varphi_{C_{y_k}}\left(V_{y_k}\left(z_{k_{j_i}}+\delta^ld\left(z_{k_{j_i}}\right)\right)-V_{y_k}\left(z_{k_{j_i}}\right)-\beta\delta^lJV_{y_k}\left(z_{k_{j_i}}\right)d\left(z_{k_{j_i}}\right)\right)>0.
	\end{equation}
	
	Since $V_{\mathcal{Y}_k}$ is continuously differentiable and $\varphi_{\mathcal{C}_{\mathcal{Y}_k}}$ is continuous, for both sides of  \eqref{ie5.8} take simultaneously $\underset{i\to\infty}{\operatorname*{lim}}$, we can get
	$$\lim_{i\to\infty}\varphi_{C_{y_k}}\left(V_{y_k}\left(z_{k_{j_i}}+\delta^ld\left(z_{k_{j_i}}\right)\right)-V_{y_k}\left(z_{k_{j_i}}\right)-\beta\delta^lJV_{y_k}\left(z_{k_{j_i}}\right)d\left(z_{k_{j_i}}\right)\right)$$
	\begin{equation}
		=\varphi_{C_{y_k}}\left(V_{y_k}\left(\hat{z}+\delta^ld(\hat{z})\right)-V_{y_k}(\hat{z})-\beta\delta^lJV_{y_k}(\hat{z})d(\hat{z})\right)\geq0,\forall l\in N,
	\end{equation}
	so
	$$V_{y_k}\left(\hat{z}+\delta^ld(\hat{z})\right)-V_{y_k}(\hat{z})-\beta\delta^lJV_{y_k}(\hat{z})d(\hat{z})\notin-intC_p$$
	\begin{equation} \label{eq5.10}
		\Rightarrow V_{y_k}\left(\hat{z}+\delta^ld(\hat{z})\right)\not\prec_{C_p}V_{y_k}(\hat{z})+\beta\delta^lJV_{y_k}(\hat{z})d(\hat{z})
	\end{equation}
	Since  \eqref{eq5.10} holds for any $l$, using Corollary \ref{co4.1}, we have: $\theta(\hat{z})\not<0$, combined with  \eqref{ie5.6}, we can obtain: $\theta(\hat{z})=0$.
\end{proof}

It follows from Theorem \ref{th3.1} and Theorem  \ref{th5.1} that the following result holds.

\begin{theorem} \label{th5.2}
	If $V_{\mathcal{Y}_k}$ is $C_{p}$-convex on $\Omega $, then the sequence $\left\{z_{k}\right\}$ generated by the improved conditional gradient algorithm with the Armijo step condition \eqref{eq4.9} converges to a weakly efficient solution of problem \eqref{wt3.4}.
\end{theorem}

\begin{remark}
	According to Definition \ref{de3.3} and Theorem \ref{th5.2}, it follows that if $V_{\mathcal{Y}_k}$ is $C_{p}$-convex on $\Omega $, and $\left\{z_{k}\right\}$ is the sequence generated by the improved conditional gradient algorithm with the Armijo step condition \eqref{eq4.9}. Then, the $\overline{x}\in D$ corresponding to every cluster point $\overline{z}$ of sequence $\left\{z_{k}\right\}$ is called a spherical weakly effective solution (or weakly Pareto solution) of problem \eqref{wt3.1}.
\end{remark}

The improved conditional gradient algorithm can calculate the minimum value of function $V_{\mathcal{Y}_k}$ on $\Omega $, indirectly obtaining the minimum value of function $y_k=F(x_k)$.

\section{Conclusion}\noindent
\setcounter{equation}{0}

In this paper, it is an important research part to define the order relationship on the sphere using the tangent space cones corresponding to points on the local sphere, and to transform the original optimization problem on the sphere into an equivalent vector optimization problem on the tangent space through exponential mapping. According to reference \cite{26}, the step size strategy for improving conditional gradient algorithm can also be studied using nonmonotonic step sizes and adaptive step sizes. It is worth noting that the auxiliary subproblem \ref{wt4.2} in the conditional gradient algorithm is not easy to implement intuitively. Whether it can be expressed as a specific expression remains to be studied. In the future, we can also study scalarization methods for spherical optimization problems.

\end{document}